\documentclass[12pt,a4paper]{article}
\usepackage[left=1in,right=1in,top=1in,bottom=1in,a4paper]{geometry}
\usepackage{amsfonts,amsmath,amsthm}
\usepackage{graphicx,dsfont}
\usepackage{pst-plot,pstricks,multido}
\usepackage{enumerate,enumitem}
\usepackage{amssymb}
\usepackage{hyperref}

\newtheorem{thm}{Theorem}[section]
\newtheorem{lem}[thm]{Lemma}

\theoremstyle{definition}

\newtheorem{remk}[thm]{Remark}

\newcommand{\kc}{\mathcal{C}}

\newcommand{\ka}{\mathcal{A}}

\makeatletter
\newcommand{\specificthanks}[1]{\@fnsymbol{#1}}
\makeatother

\title{A sharp inequality for Kendall's $\tau$ and \mbox{Spearman's $\rho$} of Extreme-Value Copulas}

\author{Thomas Mroz\thanks{Department for Mathematics, University of Salzburg, Salzburg, Austria, E-mails: noppadon.kam@gmail.com, wolfgang@trutschnig.net}\label{uni-salzburg} 
 \and
Wolfgang Trutschnig\textsuperscript{\specificthanks{1}}}

\begin{document}

\maketitle
\begin{abstract}
We derive a new (lower) inequality between Kendall's $\tau$ and Spearman's $\rho$ for two-dimensional Extreme-Value Copulas, show that this inequa\-lity is sharp in each point and conclude that the comonotonic and the product copula are the only Extreme-Value Copulas for which the well-known lower Hutchinson-Lai inequality is sharp.   
        
\end{abstract}

\section{Introduction}\label{intro}


It is well known that, on the one hand, Kendall's $\tau$ and Spearman's $\rho$ are both measures of concordance, and that, on the other hand, they quantify different aspects of the underlying dependence structure \cite{FrN}. Although a full characterization of the exact region $\Omega$ determined by all possible values of Kendall's $\tau$ and Spearman's $\rho$ was only recently provided in \cite[]{SPT}, it has been well-known since the 1950s that for (continuous) random variables $X,Y$ the value of $\vert \tau(X,Y) - \rho(X,Y)\vert$ can at most be $\frac{1}{2}$ \cite[]{Dan,DS}. For standard subfamilies of copulas like Archimedean copulas and Extreme-Value copulas the values of Kendall's $\tau$ and Spearman's $\rho$ may differ significantly less, determining the exact $\tau$-$\rho$-region might, however, be even more difficult than determining $\Omega$ has been since in subfamilies handy dense subsets (like shuffles of the minimum copula $M$ in case of $\Omega$) may be hard to find or not even exist.     

In 1990 Hutchinson and Lai conjectured that for continuous random variables $X,Y$ such that $Y$ is stochastically increasing in $X$ and vice versa (see \cite{HuLai, Nel06} for a definition and equivalent formulations) the following inequalities hold:
\begin{align}\label{ineq:HuLai}
-1 + \sqrt{1+3 \tau(X,Y)} \leq \rho(X,Y) \leq \min\Big\{\frac{3 \tau(X,Y)}{2}, 2 \tau(X,Y)-\tau^2(X,Y)\Big\}
\end{align}
In the sequel we will refer to the first inequality in (\ref{ineq:HuLai}) as lower Hutchinson-Lai inequality and to the second one as upper Hutchinson-Lai inequality. A counterexample to the upper Hutchinson-Lai inequality (concerning the $\frac{3 \tau(X,Y)}{2}$ part) can be found in \cite{Nel06}, and the lower Hutchinson-Lai inequality was disproved in \cite{MRG}. 
On the other hand, \cite{Hu} provided a variational calculus based proof of the Hutchinson-Lai inequalities for an important family of stochastically increasing (continuous) random variables $(X,Y)$ - those whose underlying copula $C_A$ is an Extreme-Value copula.
Extreme-value copulas (EVCs, for short) form an important subclass of copulas that naturally arise in various fields of application like hydrology \cite[]{Sal} 
and finance \cite[]{LS,MNE}, whenever maxima of i.i.d. sequences of random variables are considered. For more information on EVCs we refer to  \cite[]{dHR,Pick,DuSbook} and the references therein.

In the current paper we focus on EVCs and the lower Hutchinson-Lai inequality, show that it is only sharp for continuous random variables $X,Y$ that are either comonotonic or independent, prove the validity of Conjecture (C2) in \cite[]{SPT}, i.e.
\begin{align}\label{ineq:main}
\frac{3\tau(X,Y)}{2+\tau(X,Y)} \leq \rho(X,Y),
\end{align}
and then show that this inequality is sharp. As scale-invariant quantities, Kendall's $\tau$ and Spearman's $\rho$ only depend on the underlying copula, we will therefore directly work with EVCs $C_A$ and their corresponding Pickands dependence function $A$. 
Our original idea was to modify the ideas by \cite{Hu} and to tackle ineq. (\ref{ineq:main}) by tools from variational calculus. Considering that we were, however, not able to comprehend why the sets $K_\alpha$ in the proof of Theorem 4.1. in \cite[]{Hu} should be sequentially compact, we opted for another, new method of proof based on the sensitivity of $\tau$ and $\rho$ with respect to certain modifications of piecewise linear Pickands dependence functions.     
  
The rest of this paper is organized as follows: In Section 2 we introduce some notation and recall basic facts about two-dimensional extreme-value co\-pulas that will be used in the sequel. Section 3 derives ineq. (\ref{ineq:main}) with the help of two lemmata that are (in our opinion) interesting in themselves. An outlook to future work and the appendix containing a technical lemma needed in the proofs completes the paper.   

\section{Notation and preliminaries}\label{notsec}
In the sequel $\kc$ will denote the family of all two-dimen\-sional \emph{copulas}, 
$\mathcal{P}_\mathcal{C}$ the family of all \emph{doubly stochastic measures}, i.e. the family of all probability measures on $[0,1]^2$ whose marginals
are uniformly distributed on $[0,1]$; for background on copulas we refer to
\cite{DuSbook}, \cite{Nel06}, and \cite{Sem-10}. For every $C \in \kc$ the 
corresponding doubly stochastic measure will be denoted by $\mu_C$. 
Letting $d_\infty$ denote the uniform metric on $\kc$ it is well known that $(\kc,d_\infty)$
is a compact metric space. $C \in \kc$ is called \emph{extreme-value copula} (EVC) 
if there exists a copula $B \in \kc$ such that 
$$
C(x,y)=\lim_{n \rightarrow \infty} B^n\big(x^{\frac{1}{n}}, y^{\frac{1}{n}}\big)
$$
for all $x,y \in [0,1]$. It is well known that the following three conditions are equivalent \cite[]{dHR,Pick,DuSbook,Nel06}:
\begin{enumerate}
\item $C$ is an EVC.\\[-3mm]
\item $C$ is \emph{max-stable}, i.e. $C(x,y)=C^n\big(x^{\frac{1}{n}}, y^{\frac{1}{n}}\big)$ for all $n \in \mathbb{N}$ and all $x,y \in [0,1]$. \\[-3mm]
\item There exists a \emph{Pickands dependence function $A$}, i.e. a convex function $A:[0,1] \rightarrow [0,1]$ fulfilling $\max\{1-x,x\} \leq A(x)\leq 1$ for
all $x \in [0,1]$, such that
\begin{equation}\label{EVC_explicit}
C(x,y)= (xy)^{A\big(\tfrac{\ln{(x)}}{\ln{(xy)}}\big)}
\end{equation}
holds for all $x,y \in (0,1)^2$.
\end{enumerate}
 
In what follows we will only work with the convex set $\mathcal{A}$ of all Pickands dependence functions and let $\kc_{EV}$ denote the family of all extreme-value copulas. 
For every $A \in \ka$ we will write $C_A$ for the copula induced by $A$ according to (the right-hand side of) eq. (\ref{EVC_explicit}). 
$D^+A$ ($D^-A$) will denote the right-hand (left-hand) derivative of $A \in \mathcal{A}$. 
Setting $D^+A(0):=D^-A(0)$ it is straightforward to see that $D^+A:[0,1] \rightarrow [-1,1]$ is non-decreasing and right-continuous on $[0,1]$. 
The Pickands function corresponding to the minimum copula $M$ will be denoted by $A_M$, i.e. $A_M(x)=\max\{x,1-x\}$. Furthermore $\Vert \cdot \Vert_\infty$ will denote the uniform norm on $\mathcal{A}$. 
For further properties of Pickands functions and their right-hand derivative we refer to \cite{TSFS} and the references therein. 

It is well known that for every $A \in \mathcal{A}$ Spearman's $\rho$ and Kendall's $\tau$ of the corresponding Extreme-Value Copula $C_A$ can be calculated as 
\begin{align*}
	&\tau(C_A) = \int_0^1 \frac{t(1-t)}{A(t)} d(D^+A)(t),\\
	&\rho(C_A) = 12 \int_0^1 \frac{1}{(1+A(t))^2} dt - 3.
\end{align*}
In the sequel we will simply write 
$\tau(A)$ and $\rho(A)$ instead of $\tau(C_A)$ and $\rho(C_A)$.

\section{A sharp inequality between $\tau$ and $\rho$}
We are now going to prove the inequality $\rho(A)\geq \frac{3\tau(A)}{2 + \tau(A)}$ for all Pickands dependence functions $A \in \mathcal{A}$ and show that this inequality is sharp. Doing so we will first derive the result for the class of all piecewise linear Pickands dependence functions and then extend it to the full class $\mathcal{A}$ via a standard denseness and continuity argument. 

For every $\mathbf{x}=(x_1,x_2,\ldots,x_n)$ with $x_0:=0 < x_1 < x_2 < \cdots <x_n <1=:x_{n+1}$ we will let
$\mathcal{A}^{\mathbf{x}}$ denote the set of all $A \in \mathcal{A}$ 
which are linear on each of the intervals $[x_{i-1},x_{i}], i \in \{1,\ldots, n+1\}$. 
For $A \in \mathcal{A}^{\mathbf{x}}$, setting $y_i = A(x_i)$, the formulas for Kendall's $\tau$ and Spearman's $\rho$ obviously become
\begin{align}
	&\tau(A) = \sum_{i=1}^n \frac{x_i(1-x_i)}{y_i} \left(\frac{y_{i+1}-y_i}{x_{i+1}-x_i} - \frac{y_i-y_{i-1}}{x_i-x_{i+1}}\right),\label{tau.piecewise} \\
	&\rho(A) = 12 \sum_{i=1}^{n+1} \frac{x_i-x_{i-1}}{(1+y_{i-1})(1+y_i)} - 3. \label{rho.piecewise}
\end{align}

\noindent For $y_1 \in [A_M(x_1),1]$, $\mathcal{A}^{\mathbf{x}}_{y_1}$ will denote the subclass 
of all $A \in \mathcal{A}^{\mathbf{x}}$ fulfilling $A(x_1)=y_1$. $T_{x_1,y_1} \in \mathcal{A}^{\mathbf{x}}_{y_1}$ denotes the triangular Pickands function induced by $(x_1,y_1)$, i.e. the only Pickands function which is linear on $[0,x_1]$ and on $[x_1,1]$, and fulfills $T_{x_1,y_1}(x_1)=y_1$. 
For $A \in \mathcal{A}^{\mathbf{x}}_{y_1}$ and $x \in (0,x_1)$ we define the set $I_x^A$ by  
$$
I_x^A= \Big[\max\Big\{1-x,\, A(x_1)-(x_1-x)\frac{A(x_2)-A(x_1)}{x_2-x_1}\Big\},1-\frac{1-A(x_1)}{x_1}x \Big].
$$  
Obviously $I_x^A$ coincides with the set of all $y \in [0,1]$ such that there exists a (necessarily unique) Pickands function $B$ which coincides with $A$ on the interval $[x_1,1]$, which is linear on $[0,x]$ and on $[x,x_1]$ and fulfills $B(x)=y$. In the sequel we will denote this function by $\varphi_{x,y}^\mathbf{x}(A)$ and refer to $I_x^A$ as set of admissible $y$-values given $A$ and $x$. Notice that for $y=1-\frac{1-A(x_1)}{x_1}x$ we have $\varphi_{x,y}^\mathbf{x}(A)=A$ and that for each $A \in \mathcal{A}^{\mathbf{x}}_{y_1}$ the admissible set $I_x^A$ fulfills $I_x^A \subseteq I_x^{T_{x_1,y_1}}$. 
\begin{figure}[!h]
  \begin{center}
  \includegraphics[width = 8cm]{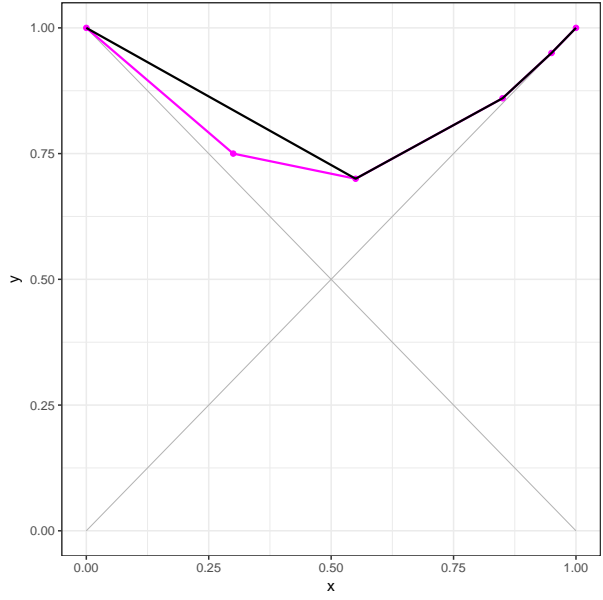}
        \caption{Example of a Pickands dependence function $A \in \mathcal{A}^{\mathbf{x}}$ with
        $\mathbf{x}=(0.05,0.55,0.85,0.95$ (black) and the Pickands function 
        $\varphi_{x,y}^\mathbf{x}(A)$ with $x=0.3,y=0.75$ (magenta).}\label{fig:phi}
  \end{center}
     \end{figure}

\noindent Define $\Delta: \mathcal{A}^2 \rightarrow \mathbb{R}$ by
\begin{equation}
	\Delta(A,B)=\rho(A)-\rho(B) - \left(\frac{3\tau(A)}{2+\tau(A)} - \frac{3\tau(B)}{2+\tau(B)} \right).
\end{equation}
The subsequent two lemmata study properties of $\Delta$ which turn out to be key for deriving a sharp lower inequality for $\tau$ and $\rho$.
\begin{lem}\label{lem:inequ}
Suppose that $A \in \mathcal{A}^{\mathbf{x}}_{y_1}$, that $x \in (0,x_1)$ and that $y \in I_x^A$. Then the following inequality holds:
\begin{equation}\label{ineq:trafo}
	\Delta(\varphi_{x,y}^\mathbf{x}(A),A) \geq \Delta(\varphi_{x,y}^{x_1}(T_{x_1,y_1}),T_{x_1,y_1}).
\end{equation}
Moreover, in case of $y_1 <1$ and $y < 1-\frac{1-A(x_1)}{x_1}x$ we have equality in (\ref{ineq:trafo}) if and only if $A=T_{x_1,y_1}$.
\end{lem}

\begin{proof}
Fix $A \in \mathcal{A}^{\mathbf{x}}_{y_1}$, $x \in (0,x_1)$, and $y \in I_x^A$.
Using equation (\ref{tau.piecewise}) and (\ref{rho.piecewise}) a straightforward calculation yields
\begin{align}
	\delta_\tau &:= \tau(\varphi_{x,y}^\mathbf{x}(A))-\tau(A) = \frac{(x_1-x + xy_1 - x_1y) (y_1-y - xy_1 + x_1y)}{(x_1-x)yy_1}, \label{eq:delta_tau}\\
	\delta_\rho &:= \rho(\varphi_{x,y}^\mathbf{x}(A))-\rho(A) = \frac{6 (x_1(1-y) - x(1-y_1))}{(1+y)(1+y_1)},\label{eq:delta_rho}
\end{align}
so both $\delta_\tau$ and $\delta_\rho$ only depend on $A$ in terms of $x_1,y_1$ and $y \in I_x^A$.
Having this we directly get
\begin{align}\label{eq:formula.Delta}
	\Delta(\varphi_{x,y}^\mathbf{x}(A),A) &= \delta_\rho - \left(\frac{3\tau(\varphi_{x,y}^\mathbf{x}(A))}{2+\tau(\varphi_{x,y}^\mathbf{x}(A))} - \frac{3\tau(A)}{2+\tau(A)} \right) \nonumber \\
	&= \delta_\rho  - \frac{6\delta_\tau}{(2+\tau(A))(2+\tau(A)+\delta_\tau)},
\end{align}
implying that $\Delta(\varphi_{x,y}^\mathbf{x}(A),A)$ decreases if $\tau(A)$ decreases. Considering $y \in I_x^A \subseteq I_x^{T_{x_1,y_1}}$ as well as the fact that 
$A \leq T_{x_1,y_1}$ implies $\tau(A) \geq \tau(T_{x_1,y_1})$ inequality (\ref{ineq:trafo}) now follows. 
The second assertion is a direct consequence of Lemma 3 in \cite{KGJT}, which says that for arbitrary $A,B \in \mathcal{A}$ with $A \leq B$ and strict inequality in at least one point $t \in (0,1)$ we have $\tau(B) < \tau(A)$ (for the sake of completeness we included the lemma in the Appendix). 
\end{proof}

\begin{lem}\label{lem:inequ2}
Suppose that $x_1 \in (0,1)$ and that $y_1 \in [A_M(x_1),1]$. 
Furthermore let $x \in (0,x_1)$ and $y \in I_x^{T_{x_1,y_1}}$ be arbitrary. Then 
\begin{align}\label{ineq:triangle}
\Delta(\varphi_{x,y}^{x_1}(T_{x_1,y_1}),T_{x_1,y_1}) \geq 0
\end{align}
holds. Moreoever, we have equality in (\ref{ineq:triangle}) if and only if $\varphi_{x,y}^{x_1}(T_{x_1,y_1})$ is a triangular Pickands dependence function too, i.e. if $\varphi_{x,y}^{x_1}(T_{x_1,y_1})=T_{x,y}$.
\end{lem}
\begin{proof}
Using $\tau(T_{x_1,y_1}) = \tfrac{1-y_1}{y_1}$ and $\rho(T_{x_1,y_1}) = \tfrac{3(1-y_1)}{1+y_1}$ equation (\ref{eq:formula.Delta}) simplifies to 
\begin{align*}
\Delta(\varphi_{x,y}^{x_1}(T_{x_1,y_1}),T_{x_1,y_1})=  \frac{6N_1 N_2 N_3} {D_1 D_2 \left(D_{3} + D_4 - D_5 + D_6 \right)}
\end{align*}
where
\begin{align*}
	N_1 &= x(1-y_1) - x_1(1-y), \,\, N_2 = x(1-y_1) - x_1(1-y)  + y_1 - y\\
	N_3 &= y(x_1+y_1) + y_1(1-x)\\
	D_1 &= 1+y, \,\,D_2= 1+y_1, \,\, D_3= x_1y(x_1-x+y(1-x_1))\\
	D_4 &= y_1(1-x)(x_1-x), \,\, D_5= 2xyy_1(1-x_1),\,\, D_6= (1-x)xy_1^2.
\end{align*}
To prove $\Delta(\varphi_{x,y}^{x_1}(T_{x_1,y_1}),T_{x_1,y_1}) \geq 0$ it suffices 
to show that both, the nominator and the denominator are non-negative for every $y \in I_x^{T_{x_1,y_1}}$, which can be done as follows: \\
(i) Obviously $N_3 >0$ so the sign of the nominator is determined by $N_1 N_2$. Considering, firstly, that $N_1 N_2$ is quadratic in $y$, secondly, that $N_1N_2$ is zero exactly for $y \in \{y_1 - (x_1-x)\frac{1-y_1}{1-x_1}, 1-\frac{1-y_1}{x_1}x\}$ (notice that the first point is not necessarily contained in $ I_x^{T_{x_1,y_1}}$), and, thirdly, that the parabola opens downwards since the coefficient of the quadratic term is given by $-x_1(1-x_1) <0$ the nominator is non-negative.\\
(ii) Since $D_1,D_2$ are obviously positive the denominator is positive if and only if 
$D_{3} + D_4 - D_5 + D_6>0$. Considering $y \in I_x^{T_{x_1,y_1}}$ there exists a unique point $b \in [0, 1-\frac{1-y_1}{x_1}x - (y_1 - (x_1-x)\frac{1-y_1}{1-x_1})]$ such that
$y=y_1 - (x_1-x)\frac{1-y_1}{1-x_1} + b$ holds. A straightforward calculation yields
\begin{align*}
D_{3} + D_4 - D_5 + D_6 &= b^2x_1(1-x_1) + b(x_1-x)(2y_1-x_1) \\
 & \quad + (x_1-x)y_1 \frac{x(1-y_1)-2x_1+y_1+1}{1-x_1}.
\end{align*}  
The first summand is obviously non-negative, the same is true for the second one because of $y_1 \geq x_1$, and the third summand is positive since 
$x(1-y_1)-2x_1+y_1+1 \geq x(1-y_1) - x_1 + 1>0$. 
Considering that the second assertion is a direct consequence of the properties of the nominator mentioned in (i) the proof is complete.
\end{proof}

Building upon the previous lemmata we can now proof the main result of this paper. It confirms the lower part of Conjecture (C2) in \cite{SPT}. 
\begin{thm}\label{thm:main}
Every Pickands dependence function $A$ fulfills $\rho(A) \geq \frac{3 \tau(A)}{2 + \tau(A)}$. The inequality is best possible in the sense that for every $u \in [0,1]$ we can find a Pickands function $A$ such that $\tau(A)=u$ and $\rho(A)=\frac{3u}{2+u}$.
\end{thm}

\begin{proof}
We first prove that the inequality holds for every piecewise linear Pickands dependence function by induction on the number of vertices: In the case of $n=1$ and $x_1 \in (0,1)$ each element 
$A$ of $\mathcal{A}^{x_1}$ is of triangular form, so we have 
$\rho(A) - \frac{3 \tau(A)}{2 + \tau(A)}=0$. In case of $n=2$ vertices and 
$x_0:=0 < x_1 < x_2 < 1$ each element $A$ of $\mathcal{A}^{(x_1,x_2)}_{y_1}$ 
fulfills $A=\varphi_{x_1,y_1}^{x_2}(T_{x_2,y_2})$. Applying inequality (\ref{ineq:triangle}) therefore yields 
\begin{align*}
\rho(A) - \frac{3 \tau(A)}{2 + \tau(A)} \geq \rho(T_{x_2,y_2}) - \frac{3 \tau(T_{x_2,y_2})}{2 + \tau(T_{x_2,y_2})}=0. 
\end{align*}
Assume now that $\rho(A) \geq \frac{3 \tau(A)}{2 + \tau(A)}$ holds for all piecewise linear Pickands functions with $n$ vertices. Suppose that   
$\mathbf{x}=(x_1,x_2,\ldots,x_n,x_{n+1})$ with $x_0:=0 < x_1 < x_2 < \cdots <x_n <x_{n+1} <1=:x_{n+2}$, fix $B \in \mathcal{A}^{\mathbf{x}}$ and set $y_1=B(x_1)$. Since
$B$ can be represented as $B=\varphi_{x_1,y_1}^{(x_2,\ldots,x_{n+1})}(A)$ for some piecewise linear $A$ with at most $n$ vertices, applying inequality (\ref{ineq:trafo}) and inequality (\ref{ineq:triangle}) yields   
\begin{align*}
\rho(B) - \frac{3 \tau(B)}{2 + \tau(B)} \geq \rho(A) - \frac{3 \tau(A)}{2 + \tau(A)} \geq 0, 
\end{align*}
which completes the proof by induction.\\
The general result for not necessarily piecewise linear Pickands functions now follows via a standard denseness and continuity argument taking into account the following two facts: (i) The mapping $\Phi:(\mathcal{A},\Vert \cdot \Vert_\infty) \longrightarrow (\mathcal{C}_{EV},d_\infty)$, 
defined by $\Phi(A)=C_A$ and the mapping $\Psi:\mathcal{C} \longrightarrow [-1,1]^2$, defined by $C \mapsto (\tau(C),\rho(C))$ are continuous. 
(ii) The family of all piecewise linear Pickands functions is dense in $(\mathcal{A},\Vert \cdot \Vert_\infty)$. The remaining assertion is a direct consequence of the identities
$\tau(T_{x_1,y_1}) = \tfrac{1-y_1}{y_1}$ and $\rho(T_{x_1,y_1}) = \tfrac{3(1-y_1)}{1+y_1}$.  
\end{proof}

\begin{remk}
Theorem \ref{thm:main} implies that the lower Hutchinson-Lai inequality  
$\rho(C_A) \geq -1 + \sqrt{1+3\tau(C_A)}$ is only sharp for the Pickands functions $A_M$ and $A_\Pi=1$, i.e. for the copulas $M$ and $\Pi$. Figure \ref{fig:omega} depicts both Hutchinson-Lai inequalities together with the boundary of the $\tau$-$\rho$-region of the full class $\mathcal{C}$ as derived in \cite{SPT} and the inequality derived in this paper.
\end{remk}

\begin{figure}[!h]
  \begin{center}
  \includegraphics[width = 10cm]{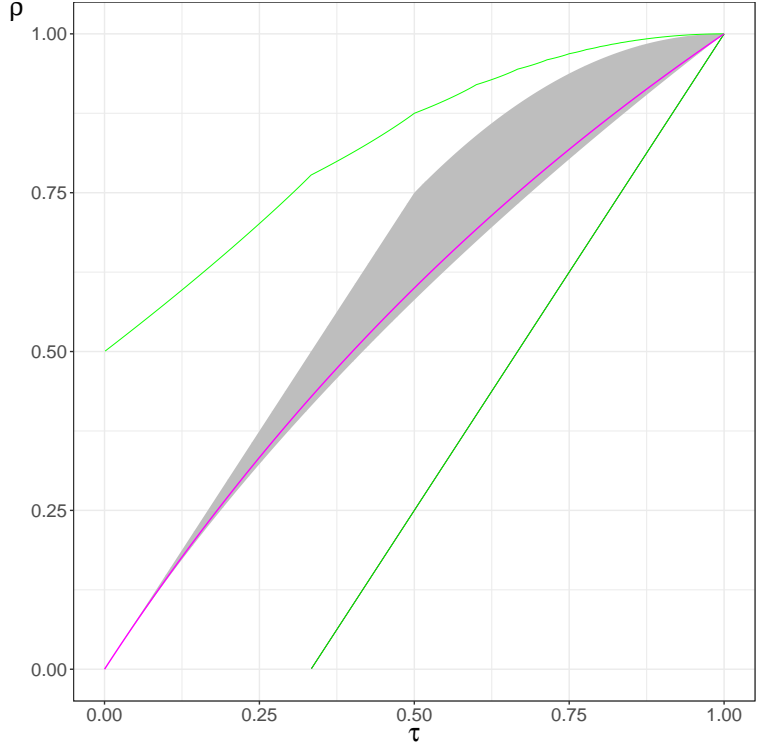}
        \caption{The region determined by the Hutchinson-Lai inequalities (gray), the boundary of the full $\tau$-$\rho$-region as recently established in \cite{SPT} (green), and the sharp inequality derived in this paper (magenta).}\label{fig:omega}
  \end{center}
     \end{figure}

\begin{remk}
Theorem \ref{thm:main} also implies that for every EVC $C_A$ Kendall's $\tau$ can not exceed Spearman's $\rho$ in general and that $\rho(C_A) \geq \frac{3 \tau(C_A)}{2+\tau(C_A)}>\tau(C_A)$ holds for all $A \in \kc_{EV} \setminus \{M,\Pi\}$. Furthermore a straightforward calculation shows that the function $f:[0,1] \longrightarrow [0,1]$, defined by $f(x)=\frac{3x}{2+x}-x$ attains its maximum at $x=\sqrt{6}-2$, which means that for $\tau(C_A)=\sqrt{6}-2$ we have
$\rho(C_A) \geq \tau(C_A) + 5- 2 \sqrt{6} \approx  \tau(C_A) +  0.101$.  
\end{remk}

\section{Conclusion and outlook}
We have shown that in the class of EVCs the lower Hutchinson-Lai inequality is only sharp for the comonotonic and the product copula, derived a new inequality, and proved that it is sharp in each point. 
We conjecture that in the class of EVCs the upper Hutchinson-Lai is not sharp either. Deriving the best-possible upper inequality will, however, be harder than deriving the sharp lower inequality has been.       

\section{Appendix}
For $A,B \in \kc$ we will write $A\succ B$ if and only if $A(t) \geq B(t)$ for every $t \in [0,1]$ with strict inequality in at least one point. 
\begin{lem}\label{preserveordertau}
For $A,B \in \mathcal{A}$ with $A \succ B $ we have $\tau(C_A) < \tau(C_B)$.
\end{lem}
\begin{proof}
If $A \geq B$ then $C_{A}\leq C_{B}$ holds and we get  
\begin{align}\label{eq:tau-strict}
\tau(C_{B})-\tau(C_{A})&=-1+4\int_{[0,1]^2}C_{B}d\mu_{C_{B}}+1-4\int_{[0,1]^2}C_{A}d\mu_{C_{A}} \nonumber \\
&=4\left(\int_{[0,1]^2}(C_{B}-C_{A})d\mu_{C_{B}}+\int_{[0,1]^2}C_{A}d\mu_{C_{B}}-\int_{[0,1]^2}C_{A}d\mu_{C_{A}}\right) \nonumber \\
&=4\left(\int_{[0,1]^2}(C_{B}-C_{A})d\mu_{C_{B}}+\int_{[0,1]^2}C_{B}d\mu_{C_{A}}-\int_{[0,1]^2}C_{A}d\mu_{C_{A}}\right)  \nonumber \\
&=4\left(\int_{[0,1]^2}(C_{B}-C_{A})d\mu_{C_{B}}+\int_{[0,1]^2}(C_{B}-C_{A})d\mu_{C_{A}}\right)\geq 0.
\end{align}
According to \cite{TSFS}, setting 
$$
L_D := \max\{x \in [0,1]: D(x)=1-x\}, \quad R_D := \min\{x \in [0,1]: D(x)=x\}
$$
for every $D \in \mathcal{A}$ the support $\textrm{Supp}(\mu_{C_D})$ of $\mu_{C_D}$ coincides with the set $\{(x,y) \in [0,1]^2: f_{L_D}(x) \leq y \leq f_{R_D}(x)\}$, whereby
$f_t:[0,1] \rightarrow [0,1]$ is defined as $f_t(x)=x^{\frac{1}{t}-1}$ for $t \in (0,1)$, and as
$f_0=0,f_1=1$ for $t \in \{0,1\}$. \\
Suppose now that $A \succ B $ holds and that $B$ does not coincide with $t \mapsto \max\{1-t,t\}$ (in which case $\tau(C_A)<1=\tau(C_B)$ is trivial). Obviously $L_A \leq L_B < \frac{1}{2}$ as well as $R_A \geq R_B > \frac{1}{2}$ and we can find some $t_0 \in (L_B,R_B)$ fulfilling
$A(t_0)>B(t_0)$. By continuity there exists some $\delta>0$ such that $A(t)>B(t)$ holds for every
$t \in [\underline{t},\overline{t}]:=[t_0 - \delta, t_0 + \delta] \subseteq (L_B,R_B) \subseteq (L_A,R_A)$. Considering 
\begin{align*}
\big\{(x,y) \in (0,1)^2: f_{\underline{t}}(x) < y < f_{\overline{t}}(x) \big\} & \subset 
\big\{ (x,y) \in (0,1)^2: C_B(x,y) > C_A(x,y)\big\} \\
&=:\{C_B > C_A\}
\end{align*}
and
$$
\big\{(x,y) \in (0,1)^2: f_{\underline{t}}(x) < y < f_{\overline{t}}(x) \big\} \subset 
\textrm{Supp}(\mu_{C_B}) \subseteq \textrm{Supp}(\mu_{C_A})
$$
shows $\mu_{C_B}(\{C_B > C_A\})>0$ and $\mu_{C_A}(\{C_B > C_A\})>0$. Hence 
$$
\int_{[0,1]^2}(C_{B}-C_{A})d \mu_{C_{B}} >0 \, \textrm{ and }
\int_{[0,1]^2}(C_{B}-C_{A})d\mu_{C_{A}} > 0
$$
follows, and applying eq. (\ref{eq:tau-strict}) yields $\tau(C_A) < \tau(C_B)$.
\end{proof}



\begin{thebibliography}{}
%







\bibitem{Dan}
Daniels, H.: Rank correlation and population models. J. Roy. Stat. Soc. B 12, 171--191 (1950)

\bibitem{DS}
Durbin, J., Stuart, A.: Inversions and rank correlation coefficients. 
J. Roy. Stat. Soc. B 12, 303--309 (1951)

\bibitem{dHR}     
de Haan, L. and Resnick, S. I.: Limit theory for multivariate sample extremes.
Z. Wahrscheinlichkeitstheorie und Verw. Gebiete 40, 317-337 (1977)


\bibitem{DuSbook} 
Durante, F., Sempi, C.: 
Principles of Copula Theory.  
CRC/Chapman \& Hall, Boca Raton (2015)










\bibitem{FrN}
Fredricks, G.A., Nelsen, R.B.: On the relationship between Spearman's rho and
Kendall's tau for pairs of continuous random variables. J. Stat. Plan. Infer. 137, 2143--2150 (2007)







\bibitem{Hu} 
H\"urlimann, W.: 
Hutchinson-Lai's conjecture for bivariate extreme value copulas.
Stat. Probabil. Lett. 61, 191-198 (2003)              

\bibitem{HuLai} 
Hutchinson, T.P., Lai, C.D.: 
  Continuous Bivariate Distributions, Emphasizing Applications. 
Rumsby Scientific, Adelaide (1990)


\bibitem{KGJT}
Kamnitui, N., Genest, C., Jaworski, P., Trutschnig, W.: 
On the size of the class of bivariate extreme-value copulas with a fixed value of Spearman's rho or Kendall's tau, submitted for publication (2018)




\bibitem{LS} 
Longin, F. and Solnik, B.: Extreme Correlation of International Equity Markets. The Journal of Finance 56, 649--676 (2001)

\bibitem{MNE} 
McNeil, A., Frey, R.,Embrechts, P.: Quantitative  risk  management. New Jersey: Princeton University Press (2005) 


\bibitem{MRG}   
Munroe, P.,  Ransford T., Genest, C.: A counterexample to a conjecture of Hutchinson and Lai.
Comptes Rendus Mathematique 348(5), 305--310 (2010)
 

\bibitem{Nel06} 
Nelsen, R.B.:
  An Introduction to Copulas. Springer Series in Statistics, New York (2006)


\bibitem{Pick}      
Pickands, J.: Multivariate extreme value distributions. In: Proceedings 43rd Session International Statistical Institute 2, 859-878 (1981)



\bibitem{Sal}      
Salvadori, G., De Michele, C., Kottegoda, N.T., and Rosso, R:
  Extremes in Nature - An Approach Using Copulas.
Springer Dordrecht (2007)


\bibitem{SPT} 
Schreyer, M., Paulin, R., Trutschnig, W.: 
On the exact region determined by Kendall's tau and Spearman's rho. 
J. Roy. Stat. Soc. B Met 79 (2), 613--633 (2017)   

\bibitem{Sem-10} 
Sempi, C.: Copul\ae: Some Mathematical Aspects.
Appl. Stoch. Model Bus. 27, 37--50 (2010)             



       
      

\bibitem{TSFS}
 Trutschnig, W., Schreyer, M., Fern\'andez S\'anchez, J.:
 Mass distributions of two-dimensional extreme-value copulas and related results. Extremes 19, 405--427 (2016)       

\end{thebibliography}
\end{document}